\documentclass{article}

\usepackage{stmaryrd}
\usepackage{lmodern}
\usepackage{diagbox}
\usepackage{mathtools}
\usepackage{amssymb,amsmath}
\usepackage{graphicx}

\usepackage{mathrsfs}

\usepackage{color}
\usepackage{amsfonts}
\usepackage{epsfig}
\usepackage{graphics}

\usepackage{lscape}
\usepackage{latexsym}

\def\qed{~~\vrule height8pt width4pt depth0pt}

\def\F2{{\mathbb F}_2}

\newtheorem{definition}{Definition}
\newtheorem{lemma}{Lemma}
\newtheorem{thm}{Theorem}

\begin{document}

\title {Combinatorial meaning of the number of the even parts in a partition of $n$ into distinct parts
}
 \author{
Jiyou Li\\
School of Mathematical Sciences\\
Shanghai Jiao Tong University\\
Shanghai, China, 200240\\
Email:~lijiyou@sjtu.edu.cn\\
Sicheng Zhao\\
School of Mathematical Sciences\\
Shanghai Jiao Tong University\\
Shanghai, China, 200240\\
Email:~zsc\_abandoned@sjtu.edu.cn
 }
\maketitle

\begin{abstract}

In a recent paper \cite{evenparts}, Andrews and Merca investigated
the number of even parts in all partitions of $n$ into distinct parts, which arise naturally from the Euler-Glaisher bijective proof. They
 obtained new combinatorial interpretations for this number by using generating functions.  We obtain a  direct combinatorial proof  in this note.

\end{abstract}
 \section{Introduction}

For a positive integer $n$, its partition $\lambda$ is a multi-set of positive integers $\lambda=\{\lambda_1,\cdots,\lambda_k\}$ such that  $\sum\limits_{i=1}^k \lambda_i=n$. Each number $\lambda_i$ is called a part of $\lambda$. For simplicity, subscripts are used to denote the multiplicity of a part in a partition. For example, $\{5_2,2_2,1\}$ denotes $\{5,5,2,2,1\}$. Also, for two partitions $\lambda=\{\lambda_1,\cdots,\lambda_k\}$ and $\mu=\{\mu_1,\cdots,\mu_t\}$, $\lambda+\mu$  is defined to be $\{\lambda_1,\cdots,\lambda_k\}\bigcup\{\mu_1,\cdots,\mu_t\}$ as a new partition.

A partition is called odd if all of its parts are odd and
distinct if all of its parts are distinct.
One of the first and most beautiful results in partition theory is  that the number of odd partitions of $n$ is equal to the number of distinct partitions of $n$. This is well known as Euler's partition identity.
  A combinatorial proof by doubling and halving naturally shows the correspondence
  and it is regarded as a typical example of elegant combinatorial proofs.

  Precisely, recall  the Glaisher's map $\varphi$ from the set of distinct partitions to the set of odd partitions is constructed as follows.
    For any distinct partition $\lambda=\{2^{k_1}(2s_1-1),\cdots,2^{k_m}(2s_m-1)\}$ with distinct pairs $(k_i,s_i)$ for $1\leq i\leq m$, $\varphi$ maps it to  $\varphi(\lambda)=\{(2s_1-1)_{2^{k_1}}\}+\cdots+\{(2s_m-1)_{2^{k_m}}\}$. For example, the Glaisher's map $\varphi$ sends the partition  $\{12,7,6,5,3\}$ to  the partition $\{7,5,3_7\}$. This elegant  map  plays an important role in our proof (Section 3).

In \cite{evenparts}, Andrews and Merca investigated
the number of even parts in all partitions of $n$ into distinct parts and they
 obtained new combinatorial interpretations for this number by using generating functions. They asked if there is any direct combinatorial proof.

 Two elegant combinatorial proofs were given by
Ballantine and  Merca \cite{combinaproof1}, and  Li  and  Wang \cite{combinaproof2} independently. In this paper, we provide a new combinatorial proof.

\begin{definition}\label{1.1}
Let $n$ be a non-negative integer. Let $N(n)$ be the number of partitions of $n$ into distinct parts and $N_{\hat{m}}(n)$ be the number of partitions of $n$ into distinct parts avoiding $m$.
\end{definition}

      For convenience, $N(n)$ and $N_{\hat{m}}(n)$ are defined to be zero when $n<0$. For simplicity, a {\bf distinct partition} of $n$ is the abbreviation for a partition of $n$ into distinct parts and similarly an {\bf odd partition} of $n$ is the abbreviation for a partition of $n$ with all parts odd.

\begin{definition}\label{1.2}
Let $n$ be a non-negative integer. Define $a(n)$ to be the number of even
parts in all distinct partitions of $n$.
\end{definition}

For instance, since $6=5+1=4+2=3+2+1$, one has $a(6)=4$.

\begin{definition}\label{1.3}
Let $n$ be a non-negative integer. We define:

1. $B_o(n)$ to be the set of partitions of $n$ into an odd number of parts in which the set of even part has only one element;

2. $B_e(n)$ to be the set of partitions of $n$ into an even number of parts in which the set of even part has only one element;

3. $b_o(n)=|B_o(n)|$, $b_e(n)=|B_e(n)|$, $b(n)=b_o(n)-b_e(n)$.
\end{definition}

Clearly, $B_o(6)=\{\{6\}, \{4, 1_2\}, \{3, 2, 1\}, \{2_3\}, \{2, 1_4\}\}$,  $B_e(6)=\{\{2_2,1_2\}\}$ and thus $b(6)=4$.

\begin{definition}\label{1.4}
Let $n$ be a non-negative integer. We define:

1. $C_o(n)$ to be the set of partitions of $n$ in which exactly one part is repeated and this part is odd;

2. $C_e(n)$ to be the set of partitions of $n$ in which exactly one part is repeated and this part is even;

3. $c_o(n)=|C_o(n)|$, $c_e(n)=|C_e(n)|$, $c(n)=c_o(n)-c_e(n).$
\end{definition}

Similarly, one checks that  $c_o(6)=5$, $c_e(6)=1$ and thus $c(6)=4$.

Beck conjectured that the number of partitions of $n$ in which the set of even parts has exactly one element equals the number of partitions of $n$ in which exactly one part is repeated, namely $b_o(n)+b_e(n)=c_o(n)+c_e(n)$. The conjecture was proved in \cite{1} and \cite{combinaproof}. Furthermore, the following two theorems in \cite{evenparts} given by  Andrews and  Merca provided more insights on this identity.
\begin{thm}\label{1.5}
 Suppose that $a(n),b(n),c(n)$ are defined as above. Then for $n \geq 1$, $a(n)=(-1)^n b(n)=c(n)$.
\end{thm}
\begin{thm}\label{1.6}
Suppose that $b_o(n),b_e(n),c_o(n),c_e(n)$ are defined as above, then we have

1. $b_o(n)=$
$\left\{
\begin{array}{cc}
c_o(n),\ & if\ n\ is\ even,    \\
c_e(n),\ & if\ n\ is\ odd; \\
\end{array}
\right.$

2. $b_e(n)=$
$\left\{
\begin{array}{cc}
c_e(n),\ & if\ n\ is\ even, \\
c_o(n),\ & if\ n\ is\ odd. \\
\end{array}
\right.$
\end{thm}We have the following example.
\begin{align*}
&    B_o(7)=\{\{3,2_2\},\{2_2,1_3\}\}, \\
 &  B_e(7)=\{\{6,1\},\{5,2\},\{4,3\},\{4,1_3\},\{3,2,1_2\},\{2_3,1\},\{2,1_5\}\},\\
  & C_o(7)=\{\{3_2,1\},\{5,1_2\},\{3,2,1_2\},\{4,1_3\},\{3,1_4\},\{2,1_5\},\{1_7\}\},\\
  & C_e(7)=\{\{2_3,1\},\{3,2_2\}\}.
\end{align*}
 The distinct partitions of $7$ are $\{7\},\{6,1\},\{5,2\},\{4,3\},\{4,2,1\}$, so that  $a(7)=0+1+1+1+2=5=-b(7)=c(7)$. Thus the theorems hold for $n=7$.

             This paper is organized as follows. We will first give proofs of Theorem \ref{1.5} and \ref{1.6} without using generating functions in Section 2 and  direct combinatorial proofs in Section 3.

\section{Proof of the main results}
\begin{lemma}\label{2.1}
Suppose that $N(n)$ and $N_{\hat{m}}(n)$ are defined as above. Then  $N_{\hat{m}}(n)=\sum\limits_{i\geq 0}(-1)^i N(n-im)$.
\end{lemma}
\begin{proof}
We first divide the distinct partitions of $n$ into two parts: the ones including $m$ and the other $N_{\hat{m}}(n)$ ones. Observing the one-to-one correspondence between distinct partitions of $n$ including $m$ as a part and distinct partitions of $n-m$ avoiding $m$, we obtain the identity $N_{\hat{m}}(n)=N(n)-N_{\hat{m}}(n-m)$. The lemma then follows directly by induction.\qed
\end{proof}

\begin{definition}\label{2.2}
Let $n$ be a non-negative integer. We define:

1. $B_o^{'}(n)$ to be the set of partitions of $n$ in which the set of even parts has only one element, and this element is repeated for an odd number of times;

2. $B_e^{'}(n)$ to be the set of partitions of $n$ in which the set of even parts has only one element, and this element is  repeated for an even number of times;

3. $b_o^{'}(n)=|B_o^{'}(n)|$, $b_e^{'}(n)=|B_e^{'}(n)|$, $b^{'}(n)=b_o^{'}(n)-b_e^{'}(n)$.
\end{definition}

Obviously  $B_o^{'}(6)=\{\{6\}, \{4, 1_2\}, \{3, 2, 1\}, \{2_3\}, \{2, 1_4\}\}$,  $B_e^{'}(6)=\{\{2_21_2\}\}$ and thus $b'(6)=4$.

The new sets $B_o^{'}(n)$ and $B_e^{'}(n)$  are closely related to $B_o(n)$ and
$B_e(n)$  due to the following lemma.
\begin{lemma}\label{2.3}
Let $B_o(n),B_e(n),B_o^{'}(n),B_e^{'}(n)$ be defined as above. Then,

1. $B_o(n)=$
$\left\{
\begin{array}{cc}
 B_o^{'}(n),\ & if\ n\ is\ even,\\
 B_e^{'}(n),\ & if\ n\ is\ odd;\\
\end{array}
\right.$\\

2. $B_e(n)=$
$\left\{
\begin{array}{cc}
B_e^{'}(n),\ & if\ n\ is\ even,\\
B_o^{'}(n),\ & if\ n\ is\ odd.\\
\end{array}
\right.$
\end{lemma}

\begin{proof}
First we prove that $B_o(n)=B_o^{'}(n)$ if $n$ is even. For a partition $\lambda=\{(2k)_t,2a_1-1,\cdots,2a_m-1\}\in B_o(n)$, the identity $n=t\cdot 2k+\sum_{i=1}^m(2a_i-1)$ requires $m$ to be even and the definition of $B_o(n)$ requires $m+t$ to be odd. Thus $t$ is odd, and $\lambda\in B_o^{'}(n)$. On the other hand, for a partition $\lambda=\{(2k)_{2s-1},2a_1-1,\cdots,2a_m-1\}\in B_o^{'}(n)$, the identity $n=(2s-1)\cdot 2k+\sum_{i=1}^m(2a_i-1)$ requires $m$ to be even. Thus $m+(2s-1)$ is odd, and $\lambda\in B_o(n)$. We conclude that $B_o(n)=B_o^{'}(n)$ for even $n$.

The remainder of the proof is similar and we omit it. \qed
\end{proof}

{\bf Proof of  Theorem \ref{1.5} and \ref{1.6}:\ }
By Lemma \ref{2.3}, Theorem \ref{1.5} and \ref{1.6} are equivalent to
the following equalities.
\begin{equation*}\label{eq1}
b_o^{'}(n)=c_o(n),
\end{equation*}
\begin{equation*}\label{eq2}
 b_e^{'}(n)=c_e(n),
\end{equation*}
\begin{equation*}\label{eq3}
 a(n)=b^{'}(n)=c(n).
\end{equation*}
We first express $b_o^{'}(n),b_e^{'}(n),c_o(n),c_e(n)$ in terms of $N(n)$ and $N_{\hat{m}}(n)$ to show the relations among them.

For a partition $\lambda=\{(2k)_{2s-1},2a_1-1,\cdots,2a_m-1\}\in B_o^{'}(n)$ with $k,s,a_i$ being positive integers, we write $\lambda=\{(2k)_{2s-1}\}+ \{2a_1-1,\cdots,2a_m-1\}$. It implies that
\begin{equation*}
    b_o^{'}(n)=\sum_{s\geq 1}\sum_{k\geq 1}N(n-(2s-1)\cdot 2k).
\end{equation*}
Similarly,
\begin{equation*}
    b_e^{'}(n)=\sum_{s\geq 1}\sum_{k\geq 1}N(n-2s\cdot 2k).
\end{equation*}

For any partition $\lambda=\{(2s-1)_t,b_1,\cdots,b_m\}\in C_o(n)$ with $t\geq 2$ and $b_i\ne 2s-1$ being distinct positive integers, we have $\lambda=\{(2s-1)_t\}+\{b_1,\cdots,b_m\}$. One obtains that
\begin{equation*}
   c_o(n)=\sum_{t\geq 2}\sum_{s\geq 1}N_{\hat{2s-1}}(n-t\cdot (2s-1)).
\end{equation*}
Similarly,
\begin{equation*}
   c_e(n)=\sum_{t\geq 2}\sum_{s\geq 1}N_{\hat{2s}}(n-t\cdot 2s).
\end{equation*}

Applying Lemma \ref{2.1} to $c_o(n)$, we have
\begin{equation*}
    c_o(n)=\sum_{t\geq 2}\sum_{s\geq 1}\sum_{i\geq 0}(-1)^i N(n-(t+i)\cdot(2s-1)).
\end{equation*}
For fixed $2s-1$, the coefficient of $N(n-l\cdot (2s-1))(l\geq 2)$ is $\sum\limits_{2\leq t \leq l}(-1)^{l-t}$. It equals 1 when $l$ is even and 0 when $l$ is odd. Thus
\begin{equation*}
    c_o(n)=\sum_{s\geq 1}\sum_{k\geq 1}N(n-2k\cdot(2s-1))=b_o^{'}(n).
\end{equation*}

Similarly, for $c_e(n)$, we have
\begin{equation*}
    c_e(n)=\sum_{s\geq 1}\sum_{k\geq 1}N(n-2k\cdot2s))=b_e^{'}(n).
\end{equation*}

Now we come to $a(n)$. By definition,
\begin{equation*}
    b^{'}(n)=b_o^{'}(n)-b_e^{'}(n)=\sum_{t\geq 1}\sum_{k\geq 1}(-1)^{t-1}N(n-t\cdot 2k)=\sum_{k\geq 1}N_{\hat{2k}}(n-2k).
\end{equation*}

Note that $\sum\limits_{k\geq 1}N_{\hat{2k}}(n-2k)$ equals the number of  pairs  (a distinct partition of $n$, an even part in it). By definition, $b^{'}(n)=a(n)$ and the proof is complete.

\section{A Combinatorial Proof}

Since our proof of Lemma \ref{2.3} is combinatorial, it suffices to establish the
following equalities combinatorially.
 \begin{equation}\label{eq1}
b_o^{'}(n)=c_o(n),
\end{equation}
\begin{equation}\label{eq2}
 b_e^{'}(n)=c_e(n),
\end{equation}
\begin{equation}\label{eq3}
 a(n)=b^{'}(n)=c(n).
\end{equation}

 Let $\varphi$ be any bijection constructed combinatorially from the set of odd partitions to the set of distinct partitions. For instance, $\varphi$  might be chosen to be the  Glaisher's map.

 We now prove (\ref{eq1}) first. Let $\lambda=\{(2k)_{2s-1},2a_1-1,\cdots,2a_m-1\}\in B_o^{'}(n)$.
  We shall construct a map $\theta$ sending $\lambda$ to a partition in $C_o(n)$.
  Let  $D=\varphi(\{2a_1-1,\cdots,2a_m-1\})$ be the image of the partition $\{2a_1-1,\cdots,2a_m-1\}$  under the map  $\varphi$. Note that $D$ is a distinct partition of $n-(2s-1)\cdot 2k$. We then send $\{(2k)_{2s-1}\}+D$ to
the target partition $\theta(\lambda)$ according the following rules:
\begin{equation*}
\left\{
\begin{array}{cc}
\{(2s-1)_{2k}\}+D,\quad & \text{if}\ 2s-1\notin D; \\
\{(2s-1)_{2k+1}\}+D\backslash\{2s-1\},\quad & \text{otherwise}. \\
\end{array}
\right.
\end{equation*}
Here $D\backslash\{2s-1\}$ denotes the multi-set (namely, the partition) obtained by removing exactly one part  $2s-1$. Obviously both of the two forms represent partitions in $C_o(n)$. One checks that this $\theta$ is injective, since two partitions must contain the same parts in $2k$ (and its multiplicity $2s-1$) and $D$  if they are sent into the same partition.

For example, let  $\varphi$  be the  Glaisher's map. Let  $\lambda=\{3,2_3,1_2\}$ be a partition of 11. Since $\varphi(\{3,1_2\})=\{3,2\}$ and $3$  appears as a part, $\theta(\lambda)=\{3_3,2\}$. Let $\lambda=\{5,3_2,2\}$ be a partition of 13. Since $\varphi(\{5,3_2\})=\{6,5\}$ and $1$  does not appear as a part,  $\theta(\lambda)=\{6,5,1_2\}$.

Conversely, for a partition $\lambda=\{(2s-1)_t,b_1,\cdots,b_m\}\in C_o(n)$ with $t\geq 2$ and $b_i\ne 2s-1$ distinct, we send it to $\{t_{2s-1},\varphi^{-1}(\{b_1,\cdots,b_m\})\}$ if  $t$ is even and $\{(t-1)_{2s-1},\varphi^{-1}(\{2s-1,b_1,\cdots,b_m\})\}$ if $t$ is odd. This map is clearly injective, and the above examples also work here. This finishes the combinatorial proof of (\ref{eq1}).

The combinatorial proof of (\ref{eq2}) is almost the  same and is left to the interested readers.

Now we turn to prove (\ref{eq3}).
By  definition we may rewrite (\ref{eq3}) as $b_o^{'}(n)=a(n)+b_e^{'}(n)$. Let $A(n)$ be the set of the pairs of  a distinct partition of $n$ and an even part in it. That is, $A(n)=\{(\lambda,2k)|\lambda\ is\ a\ distinct\ partition\ of\ n,\ 2k\in \lambda\}$. Clearly, $|A(n)|=a(n)$. We shall construct a map $\theta$ sending any partition $\lambda\in B_o^{'}(n)$ to a partition in $B_e^{'}(n)$,  or an element in $A(n)$.
  Let $\lambda=\{(2k)_{2s-1},2a_1-1,\cdots,2a_m-1\}\in B_o^{'}(n)$.
  Let  $D=\varphi(\{2a_1-1,\cdots,2a_m-1\})$ be the image of the partition under the map  $\varphi$.   We send $\lambda$ to  $\theta(\lambda)$ according the following rules:
\begin{equation}\label{eq8}
\left\{
\begin{array}{cc}
\{(2k)_{2s}\}+\varphi^{-1}(D\backslash\{2k\}),\quad &\text{if}\ 2k\in D; \\
\{(2k)_{2s-2}\}+\varphi^{-1}(D+\{2k\}),\quad &\text{if}\ 2k\notin D,\ s\not= 1; \\
 (\{2k+D\},2k) & \text{otherwise.}
\end{array}
\right.
\end{equation}
One checks that the first two forms in (\ref{eq8}) represent all  partitions in $B_e^{'}(n)$ and the third one is in $A(n)$. For two different partitions of $n$ in $B_o^{'}(n)$, if they are sent to the same partition in $B_e^{'}(n)$, then they should correspond to the same even number $2k$ and the same image $D$,  and thus also have the same multiplicity $2s-1$; if they are sent to the same element in $A(n)$, then they should be the same in partition ${2k+D}$ and corresponding even number $2k$, and thus be the same. Hence the map $\theta$ is injective.

For example,  let  $\varphi$ be the  Glaisher's map and for the  partition  $\lambda=\{5,2_3,1_3\}$ of 14, since $\varphi(\{5,1_3\})=\{5,2,1\}$ includes $2$ as a part, $\theta(\lambda)= \{2_4\}+\varphi^{-1}(\{5,1\})=\{5,2_4,1\}$.

    For the partition $\lambda=\{6_3,3_2\}$ of 24, since  $\varphi(\{3_2\})=\{6\}$ does not include $6$ as a part, $\{6_2\}+\varphi^{-1}(\{6,6\})=\{6_2,3_4\}$.

    For the partition $\lambda=\{5,3_2,2\}$ of 13, since $\varphi(\{5,3_2\})=\{6,5\}$ does  not include $2$ as a part and the multiplicity of $2$ in $\{5,3_2,2\}$ is 1, $\theta(\lambda)=(\{6,5,2\},2)$.

The converse map $\theta^{-1}$ can be obtained easily by starting from the forms in (\ref{eq8}) and
inverting the steps we took before. The details are left to the readers.


\begin{thebibliography}{}
%
%
\bibitem{1}
G.E. Andrews, Euler's partition identity and two problems of George Beck, Math.Student 86, 115--119 (2017)

\bibitem{evenparts}
G.E. Andrews and M. Merca, On the number of even parts in all partitions of $n$ into distinct parts, Annals of Combinatorics, 24, 47--54 (2020)

\bibitem{combinaproof1}
C. Ballantine and M. Merca, Combinatorial proofs of two theorems related
to the number of even parts in all partitions of $n$ into distinct parts,
 Ramanujan Journal 54, 107--112 (2021)

 \bibitem{combinaproof}
C. Ballantine and R. Bielak,  Combinatorial proofs of two Euler-Type identities due to Andrews, Annals of combinatorics, 23, 511--525 (2019)
\bibitem{combinaproof2}
R. Li  and  Andrew Y.Z. Wang, On the combinatorics of the number of even parts in all partitions with distinct parts, Ramanujan Journal 56, 721--727 (2021)
\end{thebibliography}
\end{document}